\documentclass[12pt]{amsart}
\usepackage{amsmath, amsthm,amssymb}
\usepackage{latexsym}		
\usepackage{graphicx}		
\usepackage{amssymb}
\usepackage{enumerate}
\usepackage{fancyhdr}

\newcommand{\Bb}{\mathcal{B}}

\newcommand{\Hh}{\mathcal{H}}
\newcommand{\Kk}{\mathcal{K}}
\newcommand{\Oo}{\mathcal{O}}
\newcommand{\Tt}{\mathcal{T}}

\newcommand{\TT}{\mathbb{T}}

\newcommand{\lsp}{\operatorname{span}}
\newcommand{\clsp}{\overline{\lsp}}

\numberwithin{equation}{section}

\theoremstyle{plain}
\newtheorem{theorem}{Theorem}[section]
\newtheorem{corollary}[theorem]{Corollary}

\newtheorem{proposition}[theorem]{Proposition}
\newtheorem{lemma}[theorem]{Lemma}

\theoremstyle{definition}
\newtheorem{definition}[theorem]{Definition}

\theoremstyle{remark}
\newtheorem{remark}[theorem]{Remark}

\newtheorem{notation}{Notation}

\newcommand{\MCE}{\operatorname{MCE}}
\newcommand{\ap}{\operatorname{ap}}

\begin{document}
\title{Co-universal $C^{\ast}$-algebras associated to aperiodic $k$-graphs}
\author{Sooran Kang}
\address{Sooran Kang, Aidan Sims\\ School of Mathematics and
Applied Statistics  \\
University of Wollongong\\
NSW  2522\\
AUSTRALIA} \email{sooran,asims@uow.edu.au}

\date{18 August 2011}

\author{Aidan Sims}

\begin{abstract}
We construct a representation of each finitely aligned aperiodic $k$-graph $\Lambda$ on the Hilbert
space $\Hh^{\ap}$ with basis indexed by aperiodic boundary paths in $\Lambda$. We show that the
canonical expectation on $\Bb(\Hh^{\ap})$ restricts to an expectation of the image of this
representation onto the subalgebra spanned by the final projections of the generating partial
isometries. We then show that every quotient of the Toeplitz algebra of the $k$-graph admits an
expectation compatible with this one. Using this, we prove that the image of our representation,
which is canonically isomorphic to the Cuntz-Krieger algebra, is co-universal for
Toeplitz-Cuntz-Krieger families consisting of nonzero partial isometries.
\end{abstract}


\maketitle

\section{Introduction}

A directed graph is a quadruple $(E^0, E^1, r, s)$ where $E^0$ and $E^1$ are countable sets, whose
elements are thought of as vertices and edges respectively, and $r$ and $s$ are maps from $E^1$ to
$E^0$ which indicate the directions of the edges: we think of each $e \in E^1$ as
an arrow pointing from $s(e)$ to $r(e)$. Associated to each graph there are a number of
$C^*$-algebras, the two most prominent being the Toeplitz algebra $\Tt C^*(E)$ and the
Cuntz-Krieger algebra $C^*(E)$. These are far-reaching generalizations of the Toeplitz-Cuntz
algebras $\Tt\Oo_n$ and the Cuntz algebras $\Oo_n$ \cite{Cuntz:CMP77}, and have been much studied
in recent years. One way to think of $\Tt C^*(E)$ and $C^*(E)$ is as follows: $\Tt C^*(E)$ is the
image of the natural analogue of a left-regular representation of a graph on the Hilbert space
$\ell^2(E^*)$ with basis indexed by paths in the graph, and $C^*(E)$ is the quotient of $\Tt
C^*(E)$ by its intersection with $\Kk(\ell^2(E^*))$ \cite{HongSzyma'nski:JMSJ04}.

However, this spatial viewpoint is not the traditional one. In the seminal papers \cite{KPR, KPRR}
Kumjian, Pask, Raeburn and Renault considered graphs $E$ in which $r^{-1}(v)$ is both finite and
nonempty for every $v \in E^0$. They defined $C^{\ast}(E)$ as a groupoid $C^{\ast}$-algebra, and then
 showed how $C^*(E)$ could be described as the $C^*$-algebra which is universal for a system of generators and relations
 which have come to be known as a Cuntz-Krieger family: a collection $\{p_v : v \in E^0\}$ of mutually
 orthogonal projections, and a family $\{s_e
: e \in E^1\}$ of partial isometries such that $s^*_e s_e = p_{s(e)}$ and $p_v = \sum_{r(e) = v}
s_e s^*_e$\footnote{These relations look a little different from those of \cite{KPR, KPRR}: the
roles of $r$ and $s$ are reversed. See \cite{Raeburn:cbms} for an explanation.} for each $v\in E^0$
and $e\in E^1$. The Toeplitz algebra is universal for elements $\{q_v : v \in E^0\}$ and $\{t_e : e
\in E^1\}$  such that $t^*_e t_e = q_{s(e)}$ and $q_v \ge \sum_{r(e) = v} t_e t^*_e$. It is
standard that both algebras are spanned by elements of the form $t_\mu t^*_\nu$ where $\mu,\nu$ are
directed paths in $E$ with the same source, and $t_\mu = t_{\mu_1}t_{\mu_2} \dots t_{\mu_n}$ is the
product of the generators associated to the edges occuring in the path $\mu$. This universal
groupoid $C^{\ast}$-algebra has now become standard \cite{BatesPaskEtAl:NYJM00, Katsura:TAMS04,
Raeburn:cbms}.

Two of key theorems about graph $C^*$-algebras are the uniqueness theorems. The first of these
is an Huef and Raeburn's gauge-invariant uniqueness theorem. By virtue of its universal
property, $C^*(E)$ carries an action $\gamma$, called the gauge action, of $\TT$ satisfying
$\gamma_z(s_e) = zs_e$ for all $e \in E^1$. The gauge-invariant uniqueness theorem says that
any two Cuntz-Krieger families consisting of nonzero partial isometries and carrying $\TT$-actions
compatible with the gauge action generate isomorphic $C^*$-algebras. The second uniqueness theorem,
the one in which we are interested in the present paper, is called the Cuntz-Krieger
uniqueness theorem. It says that if $E$ satisfies an appropriate aperiodicity hypothesis, then any two Cuntz-Krieger
families consisting of nonzero partial isometries generate isomorphic $C^*$-algebras.

In recent years, a number of generalizations of graph algebras have arisen, one of which is
$k$-graphs and their $C^*$-algebras, first introduced by Kumjian and Pask \cite{KP1}. A $k$-graph
is a $k$-dimensional analogue of a directed graph, and the $k$-graph $C^*$-algebra is the universal
$C^*$-algebra generated by a family of generators and relations analogous to the Cuntz-Krieger
relations for directed graphs. However, the complexity of these relations for the finitely aligned
$k$-graphs of \cite{RSY2} makes them cumbersome to verify in examples. In particular, while the
relations which determine the Toeplitz algebra remain fairly natural --- they are those which arise
in the left-regular representation --- the relation characterizing the quotient map from the
Toeplitz algebra to the Cuntz-Krieger algebra is more complicated.

Katsura's recent work on $C^*$-algebras associated to Hilbert modules \cite{K1, K2} suggests a more
elegant approach, which we call a co-universal property. The idea is that rather than specifying
the Cuntz-Krieger algebra as the universal object for a complicated set of relations (which include
the Toeplitz ones), we aim to specify it as the co-universal object for nonzero generators
satisfying the Toeplitz relations. Specifically, given a $k$-graph satisfying the hypotheses of the
Cuntz-Krieger uniqueness theorem, we aim to show that there exists a $C^*$-algebra generated by a
Toeplitz-Cuntz-Krieger family consisting of nonzero partial isometries which is the smallest such
algebra in the sense that it occurs as a quotient of any other $C^*$-algebra generated by a
Toeplitz-Cuntz-Krieger family consisting of nonzero partial isometries.

In this paper we construct such a co-universal algebra for every aperiodic finitely aligned
$k$-graph; the Cuntz-Krieger uniqueness theorem is then a corollary. Our basic technique is fairly
versatile and we believe that it is interesting in its own right. We construct a spatial
representation $\pi$ of $C^*(\Lambda)$ on a Hilbert space $\Hh^{\ap}$ for which the canonical
expectation from $\Bb(\Hh^{\ap})$ onto its diagonal subalgebra implements a map $\Phi_\pi$ on
$\pi(C^*(\Lambda))$ which sends each $\pi(s_\mu s^*_\nu)$ to $\delta_{\mu,\nu} \pi(s_\mu s^*_\mu)$.
It follows that $\Phi_\pi$ is a faithful conditional expectation. We then use algebraic arguments
to show that given any Toeplitz-Cuntz-Krieger family $\{t_\lambda : \lambda \in \Lambda\}$ for
$\Lambda$, there is an expectation $\Phi_t$ on $C^*(\{t_{\lambda}:\lambda\in\Lambda\})$ which
satisfies $\Phi_t(t_\mu t^*_\nu) = \delta_{\mu,\nu} t_\mu t^*_\mu$. From this, we deduce that the
image of the Toeplitz algebra under $\pi$ is co-universal as described above.

\section{Preliminaries}
Let $\mathbb{N}=\{0,1,\dots\}$ denote the monoid of natural numbers under addition, and regard
$\mathbb{N}^k$ an additive semigroup with identity $0$. We write $e_1,\dots,e_k$ for the generators
of $\mathbb{N}^k$, and $n_i$ for the $i$th coordinate of $n\in\mathbb{N}^k$. For
$m,n\in\mathbb{N}^k$, we say that $m\le n$ if $m_i\le n_i$ for each $i$. We write $m\vee n$ for the
coordinatewise maximum of $m$ and $n$.

A \textit{$k$-graph} $(\Lambda, d)$ consists of a countable small category $\Lambda$ together with
a \textit{degree functor} $d:\Lambda\rightarrow \mathbb{N}^k$ satisfying the \textit{factorization
property} : for every $\lambda \in \Lambda$ and $m,n\in\mathbb{N}^k$ with $d(\lambda)=m+n$, there
are unique elements $\mu,\nu\in\Lambda$ such that $\lambda=\mu\nu$ and $d(\mu)=m$, $d(\nu)=n$.
Elements $\lambda\in\Lambda$ are called \textit{paths}. For $n\in\mathbb{N}^k$, we write
${\Lambda}^n:=d^{-1}(n)$. The factorization property allows us to identify $\text{Obj}(\Lambda)$
with ${\Lambda}^0$. So we may regard the codomain and domain maps in $\Lambda$ as functions
$r,s:\Lambda\rightarrow\Lambda^0$. For $v\in\Lambda^0$, $v\Lambda=\{\lambda\in\Lambda:
r(\lambda)=v\}$ and $\Lambda v= \{\lambda\in\Lambda : s(\lambda)=v\}$.

\begin{notation}
For $\lambda\in\Lambda$ and $m\le n\le d(\lambda)$, by the factorization property we can express
$\lambda$ uniquely as $\lambda=\lambda_1\lambda_2\lambda_3$ with $d(\lambda_1)=m, d(\lambda_2)=n-m$
and $d(\lambda_3)=d(\lambda)-n$. We denote $\lambda_2$ by $\lambda(m,n)$, so
$\lambda_1=\lambda(0,m)$ and $\lambda_3=\lambda(n,d(\lambda))$.
\end{notation}

A graph morphism between two $k$-graphs $(\Lambda, d_{\Lambda})$ and $(\Gamma,d_{\Gamma})$ is a
functor $x:\Lambda\rightarrow\Gamma$ such that $d_{\Gamma}(x(\lambda))=d_{\Lambda}(\lambda)$ for
all $\lambda\in\Lambda$.

Let $k\in\mathbb{N}$ and let $m\in(\mathbb{N}\cup\{\infty\})^k$. Then $(\Omega_{k,m},d)$ is the
$k$-graph with
\[\text{Obj}(\Omega_{k,m}):=\{p\in\mathbb{N}^k:p\le m\},\]
\[\text{Mor}(\Omega_{k,m}):=\{(p,q)\in\text{Obj}(\Omega_{k,m})\times\text{Obj}(\Omega_{k,m}):p\le q\},\]
\[r(p,q):=p,\;\;s(p,q):=q\;\;d(p,q):=q-p.\]
If $x:\Omega_{k,m}\rightarrow\Lambda$ is a graph morphism and $n\in\mathbb{N}^k$ with $n\le m$, then there is a graph morphism
${\sigma}^n(x):\Omega_{k,m-n}\rightarrow\Lambda$ determined by ${\sigma}^n(x)(p,q):=x(n+p,n+q)$.
Also, if $x:\Omega_{k,m}\rightarrow\Lambda$ is a graph morphism and $\lambda\in\Lambda r(x)$, then
there is a unique graph morphism $\lambda x:\Omega_{k,m+d(\lambda)}\rightarrow\Lambda$ such that
$(\lambda x)(0,d(\lambda))=\lambda$ and ${\sigma}^{d(\lambda)}(\lambda x)=x$.

Let $(\Lambda, d)$ be a $k$-graph. For $\mu,\nu\in\Lambda$, we say that $\lambda$ is a
\textit{minimal common extension} of $\mu$ and $\nu$ if $d(\lambda)=d(\mu)\vee d(\nu)$ and
$\lambda=\mu{\mu}'=\nu{\nu}'$ for some ${\mu}',{\nu}'\in\Lambda$. We write $\MCE(\mu,\nu)$ for the
set of all minimal common extensions of $\mu$ and $\nu$. More generally, for a finite subset $F$ of
${\Lambda}$, let
\[\textstyle
    \MCE(F):=\{\lambda\in\Lambda: d(\lambda)=\bigvee_{\alpha\in F}d(\alpha)\;\text{and}\;\lambda(0,d(\alpha))=\alpha\;\text{for all}\;\alpha\in F\},
\]
and let ${\vee F:=\bigcup_{G\subset F}\MCE(G)}$.

We say that $\Lambda$ is \textit{finitely aligned} if $\text{MCE}(\lambda,\mu)$ is finite
(possibly empty) for all $\lambda,\mu\in\Lambda$.
Fix $v\in{\Lambda}^0$ and $E\subset v\Lambda$. We say that $E$ is \textit{exhaustive} if for each
$\mu\in v\Lambda$ there exists $\lambda\in E$ such that $\MCE(\mu,\lambda)\ne\emptyset$. If
$|E|<\infty$, we say that $E$ is \textit{finite exhaustive}. Define $\text{FE}(\Lambda)$ to be the
set of finite exhaustive sets of $\Lambda$ and for each $v\in{\Lambda}^0$, define
$v\text{FE}(\Lambda)$ to be the set of finite exhaustive sets whose elements all have range $v$.

\begin{definition}
We say that a $k$-graph is \textit{aperiodic} if, for all $\mu,\nu\in\Lambda$ with $s(\mu)=s(\nu)$,
there exists $\tau\in s(\mu)\Lambda$ such that $\MCE(\mu\tau,\nu\tau)=\emptyset$.
\end{definition}
\begin{remark}
There have been many versions of ``aperiodicity'' proposed for $k$-graphs. For an account of the
relationship between them, see \cite{LS,RoS1}. The version presented here first appeared in
\cite{LS}.
\end{remark}

\begin{lemma}[\cite{LS}, Lemma 4.4]\label{APT}
Let $(\Lambda,d)$ be an aperiodic finitely aligned $k$-graph, and fix $v\in{\Lambda}^0$ and a
finite subset $H$ of $\Lambda v$. Then there exists $\tau\in v\Lambda$ such that
$\MCE(\mu\tau,\nu\tau)=\emptyset$ for all distinct $\mu,\nu \in H$.
\end{lemma}

\begin{definition}
Let $(\Lambda, d)$ be a finitely aligned $k$-graph. A $k$-graph morphism
$x:\Omega_{k,m}\rightarrow\Lambda$ is called a \textit{boundary path} if for all $n\in\mathbb{N}^k$
with $n\le d(x)$ and for all $E\in x(n)\text{FE}(\Lambda)$, there exists $m\le d(x)-n$ such that
$x(n,n+m)\in E$. We write $\partial\Lambda$ for the set of all boundary paths, and for
$v\in{\Lambda}^0$, write $v(\partial\Lambda)$ for $\{x\in\partial\Lambda:r(x)=v\}$. Define the set
of \textit{aperiodic boundary paths} ${\partial\Lambda}^{{\ap}}$ by
\[{\partial\Lambda}^{{\ap}}:=\{x\in\partial\Lambda:m\ne n\Longrightarrow {\sigma}^m(x)\ne{\sigma}^n(x)\;\text{for all}\;m,n\le d(x)\}.\]
\end{definition}

\begin{lemma}\label{APP}
Let $(\Lambda,d)$ be a finitely aligned $k$-graph and let $x\in{\partial\Lambda}^{{\ap}}$.
\begin{enumerate}
\item If $m\in\mathbb{N}^k$ and $m\le d(x)$, then ${\sigma}^m(x)\in{\partial\Lambda}^{{\ap}}$.
\item If $\lambda\in\Lambda r(x)$, then $\lambda x\in{\partial\Lambda}^{{\ap}}$.
\end{enumerate}
\end{lemma}
\begin{proof}
Lemma 2.4 of \cite{LS} implies that each of $\sigma^m(x)$ and $\lambda x$ belong to
$\partial\Lambda$. So we just have to check that they are aperiodic.

For (1), observe that $p,q\le d(\sigma^m(x))$ and $\sigma^p(x)=\sigma^q(x)$ implies that $p+m,
q+m\le d(x)$ and $\sigma^{p+m}(x)=\sigma^{q+m}(x)$.

For (2), suppose that $p,q\le d(\lambda x)$ and $\sigma^p(\lambda x)=\sigma^q(\lambda x)$. For
$i\le k$,
\[
    d(\lambda)_i+d(x)_i-p_i=d(\sigma^p(\lambda x))_i
                           = d(\sigma^q(\lambda x))_i = d(\lambda)_i+d(x)_i-q_i.
\]
Hence, whenever $p_i\ne q_i$, we have $d(x)_i=\infty$. Let $p':=p-p\wedge q$ and $q':= q-p\wedge
q$. Since each $p_i' = p_i - \min\{p_i, q_i\}$, whenever $p'_i$ is nonzero, we have $p_i\ne q_i$
and hence $d(x)_i=\infty$. Similarly, $q'_i \not= 0$ implies $d(x)_i = \infty$. Now let $m :=
p'_i+((p\wedge q)\vee d(\lambda))$ and let $n := q'_i + ((p\wedge q)\vee d(\lambda))$. Then
$p,q,d(\lambda) \le m,n$. We claim that $m,n\le d(\lambda x)$. Indeed, since $p,q, d(\lambda) \le
d(\lambda x)$, we certainly have $(p \wedge q) \vee d(\lambda) \le d(\lambda x)$, and then since
$p'_i \not= 0 \implies d(x_i) = \infty$, it follows that $m \le d(\lambda x)$ also, and similarly
for $n$. We now calculate
\[
\sigma^{m - d(\lambda)}(x)
   = \sigma^m(\lambda x)
   = \sigma^{m - p}(\sigma^{p}(\lambda x)).
\]
A similar calculation gives $\sigma^{n - d(\lambda)}(x) = \sigma^{n - q}(\sigma^{q}(\lambda x))$.
Since $\sigma^p(\lambda x) = \sigma^q(\lambda x)$ by assumption and since
\[
m - p
   = m - p' - (p \wedge q)
   = ((p \wedge q) \vee d(\lambda)) - (p \wedge q)
   = n - q' - (p \wedge q)
   = n - q,
\]
it follows that $\sigma^{m - d(\lambda)}(x) = \sigma^{n - d(\lambda)}(x)$. Since $x \in
\partial\Lambda^{\ap}$, it follows that $m - d(\lambda) = n - d(\lambda)$, and hence
\[
0 = m - n = p' - q' = p - q.
\]
That is, $p,q \le d(\lambda x)$ and $\sigma^p(\lambda x) = \sigma^q(\lambda x)$ imply that $p = q$,
so $\lambda x$ is aperiodic as required.
\end{proof}

\begin{definition}\label{CK}
Let $(\Lambda,d)$ be a finitely aligned $k$-graph. A \textit{Toeplitz-Cuntz-Krieger
$\Lambda$-family} is a collection $\{t_{\lambda}:\lambda\in\Lambda\}$ of partial isometries in a
$C^{\ast}$-algebra satisfying
\begin{itemize}
\item[(TCK1)] $\{t_v:v\in{\Lambda}^0\}$ is a collection of mutually orthogonal projections;
\item[(TCK2)] $t_{\lambda}t_{\mu}=t_{\lambda\mu}$ whenever $s(\lambda)=r(\mu)$; and
\item[(TCK3)]
    $t^{\ast}_{\mu}t_\nu=\sum_{\mu\alpha=\nu\beta\in\MCE(\mu,\nu)}t_{\alpha}t^{\ast}_{\beta}$
    for all $\mu,\nu\in\Lambda$.
\end{itemize}
A \textit{Cuntz-Kriger $\Lambda$-family} is a Toeplitz-Cuntz-Krieger $\Lambda$-family
$\{t_{\lambda}:\lambda\in\Lambda\}$ which satisfies
\begin{itemize}
\item[(CK)] $\prod_{\lambda\in E}(t_v-t_{\lambda}t^{\ast}_{\lambda})=0$ for all
    $v\in{\Lambda}^0$ and $E\subset v\text{FE}(\Lambda)$.
\end{itemize}
\end{definition}

As in \cite{FMY}, given a finitely aligned $k$-graph $\Lambda$ there is a
$C^{\ast}$-algebra $\mathcal{T}C^{\ast}(\Lambda)$ called the \textit{Toeplitz algebra} of
$\Lambda$ generated by a Toeplitz-Cuntz-Krieger $\Lambda$-family $\{s_v:\lambda\in\Lambda\}$ which
is universal in the sense that given any other Topelitz-Cuntz-Krieger $\Lambda$-family
$\{t_{\lambda}:\lambda\in\Lambda\}$ in a $C^{\ast}$-algebra $B$, there exists a unique homomorphism
${\pi_t}:\mathcal{T}C^{\ast}(\Lambda)\rightarrow B$ such that $\pi_t(s_{\lambda})=t_{\lambda}$ for
every $\lambda\in\Lambda$.

\begin{proposition}\label{CCK}
Let $(\Lambda, d)$ be an aperiodic finitely aligned $k$-graph. Let $\{\xi_x : x \in
\partial\Lambda^{{\ap}}\}$ be the canonical basis for $\Hh^{\ap}:=\ell^2(\partial\Lambda^{{\ap}})$.
For $\lambda\in\Lambda$, define
\begin{equation}\label{rep}
    S^{\ap}_{\lambda}\xi_x:=\begin{cases}\xi_{\lambda x}\;\;\text{if}\;\;s(\lambda)=r(x),\\0\;\;\text{otherwise}.\end{cases}
\end{equation}
Then, $\{S^{\ap}_{\lambda}:\lambda\in\Lambda\}$ is a Cuntz-Krieger $\Lambda$-family on $\Hh^{\ap}$.
Furthermore, every $S^{\ap}_v$ is nonzero.
\end{proposition}
\begin{proof}
First, we show that $S^{\ap}_{\lambda}\ne 0$ for all $\lambda\in\Lambda$. Fix $\lambda \in
\Lambda$. Since $\Lambda$ is aperiodic, \cite[Proposition~3.6]{LS} implies that there exists
$x\in{\partial\Lambda}^{{\ap}}$ with $r(x)=s(\lambda)$. Then $\lambda
x\in{\partial\Lambda}^{{\ap}}$ by Lemma~\ref{APP}. So $S^{\ap}_{\lambda}\xi_{x}=\xi_{\lambda x}\ne
0$, which implies $S^{\ap}_{\lambda}\ne 0$.

Lemma 4.6 of \cite{Sims:IUMJ06} applied with $\mathcal{E} = \text{FE}(\Lambda)$ implies that there
is a Cuntz-Krieger $\Lambda$-family on $\ell^2(\partial\Lambda)$ satisfying (\ref{rep}). Lemma
\ref{APP} shows that $\ell^2(\partial\Lambda^{\ap})$ is an invariant subspace for the corresponding
representation and the result follows.
\end{proof}

\begin{definition}\label{dfn:csmin}
Let $(\Lambda, d)$ be an aperiodic finitely aligned $k$-graph. We define $C^*_{\min}(\Lambda)$ to
be the $C^*$-subalgebra of $\Bb(\Hh^{\ap})$ generated by $\{S^{\ap}_\lambda : \lambda \in
\Lambda\}$.
\end{definition}

The notation $C^*_{\min}(\Lambda)$ will be justified in Theorem~\ref{main}.

\section{The Co-universal algebras}

\begin{definition}
Let $(\Lambda, d)$ be a finitely aligned $k$-graph. A \textit{boolean representation} of $\Lambda$
in a $C^{\ast}$-algebra $B$ is a map $q:\lambda\mapsto q_{\lambda}$ form $\Lambda$ to $B$ such that
each $q_{\lambda}$ is a projection, and
\begin{equation}\label{booleanRep}
q_{\mu}q_{\nu}=\sum_{\gamma\in\MCE(\mu,\nu)}q_{\gamma}.
\end{equation}
\end{definition}

\begin{remark}
\begin{itemize}
\item[(i)] When $\text{MCE}(\mu,\nu)=\emptyset$, (\ref{booleanRep}) is intended to mean that
    $q_{\mu}q_{\nu}=0$.
\item[(ii)] Since $\MCE(\mu,\nu)=\MCE(\nu,\mu)$ for all $\mu,\nu$, \eqref{booleanRep} implies
    that the projections in a  boolean representation of $\Lambda$ pairwise commute.
\end{itemize}
\end{remark}

\begin{lemma}\cite[Proposition~8.6]{RS1}\label{lem1}
Let $(\Lambda, d)$ be a finitely aligned $k$-graph, let $q$ be a boolean representation of
$\Lambda$, and $F$ be a finite subset of $\Lambda$ such that $\lambda\in F$ implies $s(\lambda)\in
F$. For $\lambda\in \vee F$, define
\begin{equation}\label{Proj}
Q^{\vee F}_{\lambda}:=q_{\lambda}\prod_{\lambda\alpha\in \vee F\setminus\{\lambda\}}(q_{\lambda}-q_{\lambda\alpha}).
\end{equation}
Then, the $Q^{\vee F}_{\lambda}$ are mutually orthogonal projections and for each $\mu\in \vee F$,
\begin{equation}\label{eq1}
q_{\mu}=\sum_{\mu{\mu}'\in\vee F}Q^{\vee F}_{\mu{\mu}'}.
\end{equation}
\end{lemma}

\begin{remark}
Observe that $\mu' = s(\mu)$ indexes a term in the sum~\eqref{eq1}, so the sum includes the term
$Q^{\vee F}_\mu$.
\end{remark}

\begin{proof}[Proof of Lemma~\ref{lem1}]
The argument which establishes displayed equation on \cite[p.421]{RS1} immediately below equation
(8.5) uses only the relation
$t_{\mu}t^*_{\mu}t_{\nu}t^*_{\nu}=\sum_{\lambda\in\text{MCE}(\mu,\nu)}t_{\lambda}t^*_{\lambda}$ for
representations of product systems of graphs. So, after identifying finitely aligned product
systems of graphs over $\mathbb{N}^k$ with finitely aligned $k$-graphs as in
\cite[Example~1.4]{FS1}, an identical arguments works here.
\end{proof}

\begin{remark}
If $F\subset\Lambda$ is finite and closed under minimal common extensions, then $\vee F=F$. Hence
 $\{Q^F_{\lambda}: \lambda\in F\}$ are mutually orthogonal projections and
$q_{\mu}=\sum_{\mu{\mu}'} Q^F_{\mu{\mu}'}$ for $\mu\in F$.
\end{remark}

\begin{lemma}\label{lem2}
Let $(\Lambda, d)$ be a finitely aligned $k$-graph and let $\{t_{\lambda}:\lambda\in \Lambda\}$ be
a Toplitz-Cuntz-Krieger $\Lambda$-family. Let $q_{\alpha}=t_{\alpha}t^{\ast}_{\alpha}$ for
$\alpha\in\Lambda$. Then $q:\alpha\mapsto q_{\alpha}$ is a boolean representation of $\Lambda$. If
$\{t_{\lambda}:\lambda\in\Lambda\}$ satisfies~(CK), then for a finite exhaustive set $F\subseteq
s(\alpha)\Lambda$,
\[
    \prod_{\alpha\alpha' \in F \setminus\{\alpha\}}(q_{\alpha}-q_{\alpha{\alpha}'})=0.
\]
\end{lemma}

\begin{proof}
Multiplying~(TCK3) on the left by $t_{\mu}$ and on the right by $t^*_{\nu}$ shows that
$q_{\lambda}$ is a boolean representation of $\Lambda$. Fix $\lambda\in \Lambda$ and let $F$ be a
finite exhaustive set such that $F\subseteq s(\lambda)\Lambda$. Then $\prod_{\lambda\mu\in
F\setminus\{\lambda\}}(q_{\lambda}-q_{\lambda\mu})=\prod_{\lambda\mu\in
F\setminus\{\lambda\}}(t_{\lambda}t^{\ast}_{\lambda}-t_{\lambda\mu}t^{\ast}_{\lambda\mu})
=t_{\lambda}(\prod_{\lambda\mu\in
F\setminus\{\lambda\}}(t_{s(\lambda)}-t_{\mu}t^{\ast}_{\mu}))t^{\ast}_{\lambda}=0$ by~(CK).
\end{proof}

\begin{lemma}\label{lem3}
Let $(\Lambda,d)$ be a finitely aligned $k$-graph and let $q:\lambda\mapsto q_{\lambda}$ a boolean
representation of $\Lambda$ with each $q_{\lambda}\ne 0$.
 Let $F$ be a finite subset of $\Lambda$ which is closed under minimal common extensions.
 If $\alpha\in F$ and $\{{\alpha}'\in\Lambda\setminus\Lambda^0:\alpha{\alpha}'\in F\}$ is not exhaustive,
then the projection $Q^F_{\alpha}$ of (\ref{Proj}) is nonzero.
\end{lemma}
\begin{proof}
Suppose that $\{{\alpha}'\in\Lambda\setminus\Lambda^0:\alpha{\alpha}'\in F\}$ is not exhaustive.
Then there exist $\tau\in s({\alpha})\Lambda$ such that $\MCE({\alpha}',\tau)=\emptyset$ whenever
$\alpha{\alpha}'\in F\setminus\{\alpha\}$. Since $\MCE({\alpha}',\tau)=\emptyset$ implies that
$\MCE(\alpha{\alpha}',\alpha\tau)=\emptyset$, it follows that
$q_{\alpha{\alpha}'}q_{\alpha\tau}=\sum_{\gamma\in\MCE(\alpha{\alpha}',\alpha\tau)}q_{\gamma}=0$
for all $\alpha\alpha'\in F\setminus\{\alpha\}$. So,
$Q^F_{\alpha}q_{\alpha\tau}=q_{\alpha}\prod_{\alpha{\alpha}'\in
F\setminus\{\alpha\}}(q_{\alpha}-q_{\alpha{\alpha}'})q_{\alpha\tau}=q_{\alpha\tau}\ne 0$.
\end{proof}
For the following Proposition, we need some notation. Let $\Lambda$ be a finitely aligned
$k$-graph. For each $\lambda\in\Lambda$, define
$P^{\ap}_{\lambda}:=S^{\ap}_{\lambda}(S^{\ap}_{\lambda})^*$ where
$\{S^{\ap}_{\lambda}:\lambda\in\Lambda\}$ is the Cuntz-Krieger $\Lambda$-family of Proposition
\ref{CCK}.
\begin{proposition}\label{hm}
Let $(\Lambda,d)$ be a finitely aligned $k$-graph and let $q$ be a boolean representation of
$\Lambda$ such that $q_{\lambda}\ne 0$ for each $\lambda\in\Lambda$. Then there is a homomorphism
$\psi_q:\overline{\lsp}\{q_{\lambda}:\lambda\in\Lambda\}\rightarrow\overline{\lsp}\{P^{{\ap}}_{\lambda}:\lambda\in\Lambda\}$
satisfying $\psi_q(q_{\lambda})=P^{{\ap}}_{\lambda}$ for all $\lambda\in\Lambda$. Moreover,
$\psi_q$ is injective if and only if $\prod_{\mu\in F}(q_{\lambda}-q_{\lambda\mu})=0$ for each
$\lambda\in\Lambda$ and finite exhaustive set $F\subseteq s(\lambda)\Lambda$.
\end{proposition}
\begin{proof}
For the first assertion, fix a finite set $F\subseteq\Lambda$ and scalars $\{a_{\lambda}:\lambda\in F\}$.
Since $q_{\lambda}\mapsto P^{\ap}_{\lambda}$ preserves products and adjoints, it suffices to show that
\[\Big\Vert\sum_{\lambda\in F}a_{\lambda}P^{{\ap}}_{\lambda}\Big\Vert\;\le\;\Big\Vert\sum_{\lambda\in F}a_{\lambda}q_{\lambda}\Big\Vert.\]
For $\alpha\in \Lambda$, define
\[T^{\vee F}_{\alpha}:=P^{{\ap}}_{\alpha}\prod_{\alpha{\alpha}'\in\vee F\setminus\{\alpha\}}(P^{{\ap}}_{\alpha}-P^{{\ap}}_{\alpha{\alpha}'}),\] and let $Q^{\vee F}_{\alpha}$ be as in
(\ref{Proj}). For each $\lambda \in \vee F \setminus F$, define $a_\lambda := 0$. Then
Lemma~\ref{lem1} implies that
\[
\sum_{\lambda\in F}a_{\lambda}P^{\ap}_{\lambda}=\sum_{\lambda\in\vee F}a_{\lambda}P^{{\ap}}_{\lambda}
    =\sum_{\alpha\in\vee F}\Bigl(\sum_{\substack{\lambda\in\vee F \\ \alpha=\lambda{\lambda}'}} a_{\lambda}\Bigl)T^{\vee F}_{\alpha}
        \;\;\text{and}\;\;
\]
\[
\sum_{\lambda\in F}a_{\lambda}q_{\lambda}=    \sum_{\lambda\in\vee F}a_{\lambda}q_{\lambda}
        =\sum_{\alpha\in\vee F}\Bigl(\sum_{\substack{\lambda\in\vee F \\ \alpha=\lambda{\lambda}'}}a_{\lambda}\Bigl)Q^{\vee F}_{\alpha}.
\]
Also, Lemmas \ref{lem2}~and~\ref{lem3} imply that $\{\alpha\in\vee F:T^{\vee F}_{\alpha}\ne 0\}
\subseteq \{\alpha\in \vee F:Q^{\vee F}_{\alpha}\ne 0\}$. Thus,
\begin{equation}\label{eq2}
\Big\Vert\sum_{\lambda\in F}a_{\lambda}P^{{\ap}}_{\lambda}\Big\Vert=\max_{T^{\vee F}_{\alpha}\ne 0}\Big\vert\sum_{\lambda\in\vee F,\alpha=\lambda{\lambda}'}a_{\lambda}\Big\vert\;\le\;\max_{Q^{\vee F}_{\alpha}\ne 0}\Big\vert\sum_{\lambda\in\vee F,\alpha=\lambda{\lambda}'}a_{\lambda}\Big\vert=\Big\Vert\sum_{\lambda\in F}a_{\lambda}q_{\lambda}\Big\Vert.
\end{equation}

Now suppose that $\prod_{\mu\in F}(P^{{\ap}}_{\lambda}-P^{{\ap}}_{\lambda\mu})=0$ for each
$\lambda\in\Lambda$ and a finite exhaustive set $F\in s(\lambda)\Lambda$. Then for each finite
exhaustive set $F$, we have $\{\alpha\in F:T^F_{\alpha}\ne 0\}=\{\alpha\in F:Q^F_{\alpha}\ne 0\}$.
Thus we have equality at the second step in~(\ref{eq2}), which implies that $\psi_p$ is isometric.
To show the other direction, suppose that $\psi_p$ is injective. Fix $\lambda\in\Lambda$ and a
finite exhaustive set $F\subseteq s(\lambda)\Lambda$. Lemma \ref{lem2} implies that $\prod_{\mu\in
F}(P^{{\ap}}_{\lambda}-P^{{\ap}}_{\lambda\mu})=0$. Hence $\psi_q(\prod_{\mu\in
F}(q_{\lambda}-q_{\lambda\mu}))=\prod_{\mu\in F}(\psi_q(q_{\lambda})-\psi_q(q_{\lambda\mu}))=0$.
Since $\psi_q$ is injective, this forces  $\prod_{\mu\in F}(q_{\lambda}-q_{\lambda\mu})=0$.
\end{proof}

\begin{proposition}\label{CE}
Let $(\Lambda, d)$ be an aperiodic finitely aligned $k$-graph. Let
$\{t_{\lambda}:\lambda\in\Lambda\}$ be a Toeplitz-Cuntz-Krieger $\Lambda$-family. For $\lambda \in
\Lambda$, let $q_\lambda := t_\lambda t^*_\lambda$. Then for a finite subset $F$ of $\Lambda$ and
collection $\{a_{\mu,\nu}:\mu,\nu\in F\}$ of scalars,
\begin{equation}\label{Claim1}
\Big\Vert\sum_{\mu\in F}a_{\mu,\mu}q_{\mu}\Big\Vert\;\le\;\Big\Vert\sum_{\mu,\nu\in F}a_{\mu,\nu}t_{\mu}t^{\ast}_{\nu}\Big\Vert.
\end{equation}
Moreover, there exists a conditional expectation
$\Phi_t:C^{\ast}(\{t_{\lambda}:\lambda\in\Lambda\})\rightarrow\overline{\lsp}\{q_{\lambda}:\lambda\in\Lambda\}$
satisfying $\Phi_t(t_{\mu}t^{\ast}_{\nu})=\delta_{\mu,\nu}q_{\mu}$. In particular, if
$\pi_t:\mathcal{T}C^*(\Lambda)\rightarrow C^*(\{t_{\lambda}:\lambda\in\Lambda\})$ and
$\pi_{S^{\ap}}:\mathcal{T}C^*(\Lambda)\rightarrow C^*(\{S^{\ap}_{\lambda}:\lambda\in\Lambda\})$ are
the homomorphisms induced by the universal property of $\mathcal{T}C^*(\Lambda)$, then
\begin{equation}\label{exp}
{\psi}_q\circ\Phi_t\circ\pi_t=\Phi_{S^{{\ap}}}\circ\pi_{S^{{\ap}}}.
\end{equation}
\end{proposition}

\begin{proof}
To prove \eqref{Claim1}, fix a finite subset $F\subseteq\Lambda$ and scalars
$\{a_{\mu,\nu}:\mu,\nu\in F\}$. Assume without loss of generality that $F$ is closed under minimal
common extensions. Define
\[F'=\bigcup_{\mu,\nu\in F}\{\lambda{\beta}',\lambda{\delta}'\;:\;\lambda=\mu\beta=\nu\delta\in\MCE(\mu,\nu)\;\;\text{and}\;\;\beta{\beta}'=\delta{\delta}'\in\MCE(\beta,\delta)\}.\]
Then $F'$ is finite since $F$ is finite and each $\MCE(\beta,\delta)$ is finite. Let $\overline{F}=F\cup F'$.

For each $\lambda\in F$, let $B^{\vee \overline{F}}_{\lambda}=\{\lambda'\in
s(\lambda)\Lambda\setminus\{s(\lambda)\}:\lambda{\lambda}'\in \vee\overline{F}\}$. For each
$\lambda\in F$, if $B^{\vee\overline{F}}_{\lambda}$ is not exhaustive, then there exists
${\alpha}_1\in s(\lambda)\Lambda$ such that $\MCE(\mu,{\alpha}_1)=\emptyset$ for all $\mu\in
B^{\vee\overline{F}}_{\lambda}$. Fix $\alpha \in s(\alpha_1)\Lambda$ such that for any $i \le k$
satisfying $d(\alpha)_i < \max_{\mu \in \vee\overline{F}} d(\mu)_i$, we have
$s(\alpha)\Lambda^{e_i} = \emptyset$. Define $\alpha^\lambda := \alpha_1\alpha$. Then $\MCE(\mu,
\alpha^\lambda) = \emptyset$ for each $\mu \in F$ by choice of $\alpha_1$; and our choice of
$\alpha$ ensures that
\begin{equation}\label{eq:long enough}
\big(\mu \in \vee\overline{F}\text{ and }\MCE(\lambda\alpha^\lambda, \mu) \not= \emptyset\big)\; \implies\;
    \lambda\alpha^\lambda = \mu\mu'\text{ for some } \mu' \in \Lambda
\end{equation}
%
%
%

Let $G=\{\epsilon\in\Lambda:\lambda{\alpha}^{\lambda}=\mu\epsilon\;\text{for some}\;\lambda,\mu\in
\vee\overline{F}\}$. By Lemma~\ref{APT}, for each $v\in s(G) := \{s(\epsilon) : \epsilon \in G\}$,
there exists $\tau_v\in v\Lambda$  such that $\MCE(\epsilon\tau_v,\delta\tau_v)=\emptyset$ for all
distinct $\epsilon,\delta\in G v$. For $\lambda\in \vee\overline{F}$, define
$\tau^{\lambda}:=\tau_{s(\alpha^{\lambda})}$. For $\lambda\in F$, we define
\begin{equation}\label{phi1}
{\phi}^{\vee\overline{F}}_{\lambda}:=\begin{cases}q_{\lambda{\alpha}^{\lambda}{\tau}^{\lambda}}\;\;\text{if}\;\;B^{\vee\overline{F}}_{\lambda}\;\text{is not exhaustive}\\Q^{\vee\overline{F}}_{\lambda}\;\;\;\;\;\;\;\text{otherwise.}\end{cases}
\end{equation}
By definition, each ${\phi}^{\vee\overline{F}}_{\lambda}\le Q^{\vee\overline{F}}_{\lambda}$. So
since the $Q^{\vee\overline{F}}_{\lambda}$ are mutually orthogonal, the
${\phi}^{\vee\overline{F}}_{\lambda}$ are mutually orthogonal. Hence, for $\lambda\in F$
\begin{equation}\label{eqA}
\Big\Vert\sum_{\mu,\nu\in \vee\overline{F}}a_{\mu,\nu}t_{\mu}t^{\ast}_{\nu}\Big\Vert
\;\ge\;\Big\Vert\sum_{\lambda\in F}{\phi}^{\vee\overline{F}}_{\lambda}\Bigl(\sum_{\mu,\nu\in \vee\overline{F}}a_{\mu,\nu}t_{\mu}t^{\ast}_{\nu}\Bigl){\phi}^{\vee\overline{F}}_{\lambda}\Big\Vert.
\end{equation}

For $\lambda,\mu,\nu\in F$, we claim that
\begin{equation}\label{phi2}
{\phi}^{\vee\overline{F}}_{\lambda}t_{\mu}t^{\ast}_{\nu}{\phi}^{\vee\overline{F}}_{\lambda}=\begin{cases}{\phi}^{\vee\overline{F}}_{\lambda}\;\;\text{if $\mu=\nu$ and $\lambda=\mu{\lambda}'$ for some ${\lambda}'$}\\0\;\;\;\;\;\;\;\text{otherwise.}\end{cases}
\end{equation}
To prove \eqref{phi2}, we consider a number of cases separately.

\textbf{Case~1}: $\mu=\nu$. We first show that \eqref{phi2} holds if either $\lambda=\mu\lambda'$
or $\MCE(\lambda,\mu)=\emptyset$. If $\lambda=\mu{\lambda}'$, then
${\phi}^{\vee\overline{F}}_{\lambda}\le Q^{\vee\overline{F}}_{\lambda}\le q_{\lambda}\le q_{\mu}$,
and hence
${\phi}^{\vee\overline{F}}_{\lambda}t_{\mu}t^{\ast}_{\nu}{\phi}^{\vee\overline{F}}_{\lambda}
={\phi}^{\vee\overline{F}}_{\lambda}q_{\mu}{\phi}^{\vee\overline{F}}_{\lambda}={\phi}^{\vee\overline{F}}_{\lambda}$.
And if $\MCE(\mu,\lambda)=\emptyset$, then $q_{\mu}q_{\lambda}=0$; and hence the identities
$q_{\lambda{\alpha}^{\lambda}{\tau}^{\lambda}}q_{\mu}q_{\lambda{\alpha}^{\lambda}{\tau}^{\lambda}}=q_{\lambda{\alpha}^{\lambda}{\tau}^{\lambda}}q_{\lambda}q_{\mu}q_{\lambda{\alpha}^{\lambda}{\tau}^{\lambda}}=0$
and
$Q^{\vee\overline{F}}_{\lambda}q_{\mu}Q^{\vee\overline{F}}_{\lambda}=Q^{\vee\overline{F}}_{\lambda}q_{\lambda}q_{\mu}Q^{\vee\overline{F}}_{\lambda}=0$
give ${\phi}^{\vee\overline{F}}_{\lambda}q_{\mu}{\phi}^{\vee\overline{F}}_{\lambda}=0$. So to
complete Case~1, we suppose that $\lambda\ne\mu\lambda'$ and that $\MCE(\mu,\lambda)\ne\emptyset$;
we must show that $\phi^{\vee \overline{F}}_\lambda q_\mu \phi^{\vee \overline{F}}_\lambda = 0$. We
consider two subcases.

Case~1a: $\mu = \nu$ and $B^{\vee\overline{F}}_{\lambda}$ is not exhaustive. Then
${\phi}^{\vee\overline{F}}_{\lambda}=q_{\lambda{\alpha}^{\lambda}{\tau}^{\lambda}}$, so since
$q_{\lambda}\ge q_{\lambda\alpha^{\lambda}\tau^{\lambda}}$,
\begin{equation*}
{\phi}^{\vee\overline{F}}_{\lambda}q_{\mu}{\phi}^{\vee\overline{F}}_{\lambda}
=q_{\lambda{\alpha}^{\lambda}{\tau}^{\lambda}}q_{\mu}q_{\lambda}q_{\lambda{\alpha}^{\lambda}{\tau}^{\lambda}}
=\sum_{\gamma\in\MCE(\mu,\lambda)}q_{\lambda{\alpha}^{\lambda}{\tau}^{\lambda}}q_{\gamma}q_{\lambda{\alpha}^{\lambda}}q_{\lambda{\alpha}^{\lambda}{\tau}^{\lambda}}.
\end{equation*}
By choice of $\alpha^{\lambda}$, $\MCE({\lambda}',{\alpha}^{\lambda})=\emptyset$ for all
${\lambda}'\in B^{\vee\overline{F}}_{\lambda}$. Hence
$\MCE(\gamma,\lambda\alpha^{\lambda})=\MCE(\lambda{\lambda}',\lambda{\alpha}^{\lambda})=\emptyset$.
Thus, $q_{\gamma}q_{\lambda{\alpha}^{\lambda}}=0$, giving
${\phi}^{\vee\overline{F}}_{\lambda}q_{\mu}{\phi}^{\vee\overline{F}}_{\lambda}=0$. This completes
Case~1a.

Case~1b: $\mu = \nu$ and $B^{\vee\overline{F}}_{\lambda}$ is exhaustive. Then
${\phi}^{\vee\overline{F}}_{\lambda}=Q^{\vee\overline{F}}_{\lambda}$. Whenever
$\beta,\beta\beta'\in\vee \overline{F}$, we have $Q^{\vee\overline{F}}_{\beta}\le
q_{\beta}-q_{\beta{\beta}'}$ and $(q_{\beta}-q_{\beta{\beta}'})\perp q_{\beta{\beta}'}$. Since
$\lambda,\mu\in F$ and $F$ is closed under minimal common extensions, $\MCE(\lambda,\mu)\subset
F\subset \vee\overline{F}$. So,
\begin{equation*}\begin{split}
Q^{\vee\overline{F}}_{\lambda}q_{\mu}Q^{\vee\overline{F}}_{\lambda}
&=Q^{\vee\overline{F}}_{\lambda}q_{\mu}q_{\lambda}\prod_{\lambda\alpha\in \vee\overline{F}\setminus\{\lambda\}}(q_{\lambda}-q_{\lambda\alpha})\\
&=Q^{\vee\overline{F}}_{\lambda}\sum_{\lambda{\lambda}'\in\MCE(\lambda,\mu)}q_{\lambda{\lambda}'}\prod_{\lambda\alpha\in \vee\overline{F}\setminus\{\lambda\}}(q_{\lambda}-q_{\lambda\alpha})=0.
\end{split}
\end{equation*}
This complete Case~1b.

\textbf{Case~2}: $\mu \not= \nu$. Again we consider two subcases.

Case~2a: $\mu\ne\nu$ and $B^{\vee\overline{F}}_{\lambda}$ is exhaustive. Then
${\phi}^{\vee\overline{F}}_{\lambda}=Q^{\vee\overline{F}}_{\lambda}$. We calculate:
\[
t_{\lambda}t^{\ast}_{\lambda}t_{\mu}t^{\ast}_{\nu}t_{\lambda}t^{\ast}_{\lambda}
    =q_{\lambda}q_{\mu}t_{\mu}t^{\ast}_{\nu}q_{\nu}q_{\lambda}
    =\sum_{\substack{\lambda\alpha=\mu\beta\in\MCE(\lambda,\mu) \\ \lambda\tau=\nu\delta\in\MCE(\lambda,\nu)}}q_{\mu\beta}t_{\mu}t^{\ast}_{\nu}q_{\nu\delta}.
\]
Applying (TCK3) and the definition of the $q_\lambda$ to each term in the sum gives
\[
t_{\lambda}t^*_{\lambda}t_{\mu}t^*_{\nu}t_{\lambda}t^*_{\lambda}
    =\sum_{\substack{\lambda\alpha=\mu\beta\MCE(\lambda,\mu)\\ \lambda\tau=\nu\delta\in\MCE(\lambda,\nu)\\ \beta{\beta}'=\delta{\delta}'\in\MCE(\beta,\delta)}}
        t_{\lambda\alpha{\beta}'}t^{\ast}_{\lambda\tau{\delta}'}.
\]

Thus,
\begin{equation*}\begin{split}
Q^{\vee\overline{F}}_{\lambda}t_{\mu}t^{\ast}_{\nu}Q^{\vee\overline{F}}_{\lambda}
&=q_{\lambda}\prod_{\lambda{\lambda}'\in \vee\overline{F}\setminus\{\lambda\}}(q_{\lambda}-q_{\lambda{\lambda}'})t_{\mu}t^{\ast}_{\nu}q_{\lambda}\prod_{\lambda{\lambda}'\in \vee\overline{F}}(q_{\lambda}-q_{\lambda{\lambda}'})\\
&=\prod_{\lambda{\lambda}'\in \vee\overline{F}\setminus\{\lambda\}}(q_{\lambda}-q_{\lambda{\lambda}'})\Bigl(\sum_{\substack{\lambda\alpha=\mu\beta\in\MCE(\lambda,\mu)\\ \lambda\tau=\nu\delta\in\MCE(\lambda,\nu)\\ \beta{\beta}'=\delta{\delta}'\MCE(\beta,\delta)}}
t_{\lambda\alpha{\beta}'}t^{\ast}_{\lambda\tau{\delta}'}\Bigl)\prod_{\lambda{\lambda}'\in \vee\overline{F}}(q_{\lambda}-q_{\lambda{\lambda}'})\\
&=\prod_{\lambda{\lambda}'\in \vee\overline{F}\setminus\{\lambda\}}(q_{\lambda}-q_{\lambda{\lambda}'})\Bigl(\sum_{\substack{\lambda\alpha=\mu\beta\in\MCE(\lambda,\mu)\\ \lambda\tau=\nu\delta\in\MCE(\lambda,\nu)\\ \beta{\beta}'=\delta{\delta}'\in\MCE(\beta,\delta)}}
q_{\lambda\alpha{\beta}'}t_{\lambda\alpha{\beta}'}t^{\ast}_{\lambda\tau{\delta}'}q_{\lambda\tau{\delta}'}\Bigl)\prod_{\lambda{\lambda}'\in \vee\overline{F}}(q_{\lambda}-q_{\lambda{\lambda}'}).
\end{split}
\end{equation*}
Fix $\lambda\alpha=\mu\beta\in\MCE(\lambda,\mu)$, $\lambda\tau=\nu\delta\in\MCE(\lambda,\nu)$ and $\beta{\beta}'=\delta{\delta}'\in\MCE(\beta,\delta)$.
If either $d(\alpha)\ne 0$ or $d(\tau)\ne 0$, then
\[
    \prod_{\lambda{\lambda}'\in \vee\overline{F}\setminus\{\lambda\}}(q_{\lambda}-q_{\lambda{\lambda}'})
        q_{\lambda\alpha{\beta}'}t_{\lambda\alpha{\beta}'}t^{\ast}_{\lambda\tau{\delta}'}q_{\lambda\tau{\delta}'}
        \prod_{\lambda{\lambda}'\in \vee\overline{F}\setminus\{\lambda\}}(q_{\lambda}-q_{\lambda{\lambda}'})=0
\]
because
$\prod_{\lambda{\lambda}'\in\vee\overline{F}}(q_{\lambda}-q_{\lambda{\lambda}'})q_{\lambda\alpha{\beta}'}\le(q_{\lambda}-q_{\lambda\alpha})q_{\lambda\alpha{\beta}'}=0$
and similarly on the right. So suppose that $d(\alpha)=d(\tau)=0$, then $\mu=\lambda(0,d(\mu))$ and
$\nu=\lambda(0,d(\nu))$. So $\mu\ne \nu$ implies $d(\beta')-d(\delta')=d(\nu)-d(\mu)\ne 0$, whence
one of $d(\beta')$ and $d(\delta')$ is nonzero. Since
$\lambda\beta',\lambda\delta'\in\vee\overline{F}$ by construction, this forces
\[\begin{split}
\prod_{\lambda{\lambda}'\in \vee\overline{F}\setminus\{\lambda\}}(q_{\lambda}-q_{\lambda{\lambda}'})&
    q_{\lambda\alpha{\beta}'}t_{\lambda\alpha{\beta}'}t^{\ast}_{\lambda\tau{\delta}'}q_{\lambda\tau{\delta}'}
    \prod_{\lambda{\lambda}'\in \vee\overline{F}\setminus\{\lambda\}}(q_{\lambda}-q_{\lambda{\lambda}'})\\
    &{}=\prod_{\lambda{\lambda}'\in\vee\overline{F}\setminus\{\lambda\}}(q_{\lambda}-q_{\lambda{\lambda}'})q_{\lambda\beta'} t_{\lambda\beta'} t^*_{\lambda\delta'}q_{\lambda\delta'}
        \prod_{\lambda\lambda'\in\vee\overline{F}\setminus\{\lambda\}}(q_{\lambda}-q_{\lambda\lambda'})=0
\end{split}\]
as above. This completes Case 2a.

Case~2b: $\mu\ne \nu$ and $B^{\vee\overline{F}}_{\lambda}$ is not exhaustive. Then,
\begin{equation*}\begin{split}
{\phi}^{\vee\overline{F}}_{\lambda}t_{\mu}t^{\ast}_{\nu}{\phi}^{\vee\overline{F}}_{\lambda}
&=q_{\lambda{\alpha}^{\lambda}{\tau}^{\lambda}}q_{\mu}t_{\mu}t^{\ast}_{\nu}q_{\nu}q_{\lambda{\alpha}^{\lambda}{\tau}^{\lambda}}\\
&=\sum_{\substack{\lambda{\alpha}^{\lambda}{\tau}^{\lambda}\eta=\mu\gamma\in\MCE(\lambda{\alpha}^{\lambda}{\tau}^{\lambda},\mu)\\ \lambda{\alpha}^{\lambda}{\tau}^{\lambda}\zeta=\nu\delta\in\MCE(\lambda{\alpha}^{\lambda}{\tau}^{\lambda},\nu)}}
q_{\lambda{\alpha}^{\lambda}{\tau}^{\lambda}\eta}t_{\mu}t^{\ast}_{\nu}q_{\lambda{\alpha}^{\lambda}{\tau}^{\lambda}\zeta}\\
&=\sum_{\substack{\lambda{\alpha}^{\lambda}{\tau}^{\lambda}\eta=\mu\gamma\in\MCE(\lambda{\alpha}^{\lambda}{\tau}^{\lambda},\mu)\\
\lambda{\alpha}^{\lambda}{\tau}^{\lambda}\zeta=\nu\delta\in\MCE(\lambda{\alpha}^{\lambda}{\tau}^{\lambda},\nu)}}
t_{\lambda{\alpha}^{\lambda}{\tau}^{\lambda}\eta}t^{\ast}_{\lambda{\alpha}^{\lambda}{\tau}^{\lambda}\eta}t_{\mu}t^{\ast}_{\nu}
t_{\lambda{\alpha}^{\lambda}{\tau}^{\lambda}\zeta}t^{\ast}_{\lambda{\alpha}^{\lambda}{\tau}^{\lambda}\zeta}.
\end{split}
\end{equation*}

By~\eqref{eq:long enough}, if $\lambda\alpha^{\lambda}\ne\mu\epsilon'$ or
$\lambda\alpha^{\lambda}\ne\nu\delta'$, then
$\phi^{\vee\overline{F}}_{\lambda}t_{\mu}t^*_{\nu}\phi^{\vee\overline{F}}_{\lambda}=0$. So suppose
that $\lambda{\alpha}^{\lambda}=\mu{\epsilon}'$ and $\lambda{\alpha}^{\lambda}=\nu{\delta}'$. Then,
\[t^{\ast}_{\lambda{\alpha}^{\lambda}{\tau}^{\lambda}\eta}t_{\mu}t^{\ast}_{\nu}t_{\lambda{\alpha}^{\lambda}{\tau}^{\lambda}\zeta}
=t^{\ast}_{{\epsilon}'{\tau}^{\lambda}\eta}t^{\ast}_{\mu}t_{\mu}t^{\ast}_{\nu}t_{\nu}t_{{\delta}'{\tau}^{\lambda}\zeta}
=t^{\ast}_{{\epsilon}'{\tau}^{\lambda}\eta}t_{{\delta}'{\tau}^{\lambda}\zeta}.\]
By choice of ${\tau}^{\lambda}$, for distinct ${\epsilon}',{\delta}'$ such that
$\lambda{\alpha}^{\lambda}=\mu{\epsilon}'=\nu{\delta}'$,
 we have $\MCE({\epsilon}'{\tau}^{\lambda},{\delta}'{\tau}^{\lambda})=\emptyset$, which implies that
 $\MCE({\epsilon}'{\tau}^{\lambda}\eta,{\delta}'{\tau}^{\lambda}\zeta)=\emptyset$ for any $\eta,\zeta \in \Lambda$.
Thus,
$t^{\ast}_{\lambda{\alpha}^{\lambda}{\tau}^{\lambda}\eta}t_{\mu}t^{\ast}_{\nu}t_{\lambda{\alpha}^{\lambda}{\tau}^{\lambda}\zeta}
=t^{\ast}_{{\epsilon}'{\tau}^{\lambda}\eta}t_{{\delta}'{\tau}^{\lambda}\zeta}=0$, forcing
${\phi}^{\vee\overline{F}}_{\lambda}t_{\mu}t^{\ast}_{\nu}{\phi}^{\vee\overline{F}}_{\lambda}=0$.
This completes Case~2b, establishing~\eqref{phi2}.

Now, to establish \eqref{Claim1}, we first show that
\begin{equation}\label{eq3}
\{\lambda\in\vee\overline{F}\;:\;{\phi}^{\vee\overline{F}}_{\lambda}\ne 0\}=\{\lambda\in\vee\overline{F}\;:\;Q^{\vee\overline{F}}_{\lambda}\ne 0\}.
\end{equation}

If $B^{\vee\overline{F}}_{\lambda}$ is exhaustive, then
${\phi}^{\vee\overline{F}}_{\lambda}=Q^{\vee\overline{F}}_{\lambda}$, which implies that
${\phi}^{\vee\overline{F}}_{\lambda}\ne 0\Leftrightarrow Q^{\vee\overline{F}}_{\lambda}\ne 0$; and
if $B^{\vee\overline{F}}_{\lambda}$ is not exhaustive, then
${\phi}^{\vee\overline{F}}_{\lambda}=q_{\lambda{\alpha}^{\lambda}{\tau}^{\lambda}}=q_{\lambda{\alpha}^{\lambda}{\tau}^{\lambda}}Q^{\vee\overline{F}}_{\lambda}$,
and since $q_{\lambda{\alpha}^{\lambda}{\tau}^{\lambda}}\ne 0$ by hypothesis, it follows that both
$\phi^{\vee\overline{F}}_{\lambda}$ and $Q^{\vee\overline{F}}_{\lambda}$ are nonzero. This gives
\eqref{eq3}.

Define $a_{\mu,\nu} := 0$ for $\mu,\nu\in\vee\overline{F}\setminus F$. Then
\begin{equation*}\begin{split}
\Big\Vert\sum_{\mu,\nu\in F}a_{\mu,\nu}t_{\mu}t^{\ast}_{\nu}\Big\Vert
&=\Big\Vert\sum_{\mu,\nu\in \vee\overline{F}}a_{\mu,\nu}t_{\mu}t^{\ast}_{\nu}\Big\Vert\\
&\ge\;\Big\Vert\sum_{\lambda\in F}{\phi}^{\vee\overline{F}}_{\lambda}\Bigl(\sum_{\mu,\nu\in \vee\overline{F}}a_{\mu,\nu}t_{\mu}t^{\ast}_{\nu}\Bigl){\phi}^{\vee\overline{F}}_{\lambda}\Big\Vert\;\;\text{by (\ref{eqA})}\\
&=\;\Big\Vert\sum_{\lambda\in F}{\phi}^{\vee\overline{F}}_{\lambda}\Bigl(\sum_{\mu,\nu\in F}a_{\mu,\nu}t_{\mu}t^{\ast}_{\nu}\Bigl){\phi}^{\vee\overline{F}}_{\lambda}\Big\Vert\\
&=\;\Big\Vert\sum_{\lambda\in F}\Bigl(\sum_{\mu\in F}a_{\mu,\mu}t_{\mu}t^{\ast}_{\mu}\Bigl){\phi}^{\vee\overline{F}}_{\lambda}\Big\Vert\;\;\text{by (\ref{phi2})}\\
&=\max_{{\phi}^{\vee\overline{F}}_{\lambda}\ne 0}\Big\vert\sum_{\substack{\mu\in F \\ \lambda=\mu{\mu}'\in F}}a_{\mu,\mu}\Big\vert.
\end{split}\]

Moreover,
\[\begin{split}
\Big\Vert\sum_{\mu\in F}a_{\mu,\mu}q_{\mu}\Big\Vert
&=\Big\Vert\sum_{\mu\in\vee\overline{F}}a_{\mu,\mu}q_{\mu}\Big\Vert\\
&=\;\Big\Vert\sum_{\lambda\in\vee\overline{F}}\Bigl(\sum_{\mu\in\vee\overline{F}}a_{\mu,\mu}t_{\mu}t^{\ast}_{\mu}\Bigl)Q^{\vee\overline{F}}_{\lambda}\Big\Vert\;\;\text{by (\ref{eq1})}\\
&=\max_{Q^{\vee\overline{F}}_{\lambda}\ne 0}\Big\vert\sum_{\substack{\mu\in\vee\overline{F} \\ \lambda=\mu{\mu}'\in\vee\overline{F}}}a_{\mu,\mu}\Big\vert\\
&=\max_{Q^{\vee\overline{F}}_{\lambda}\ne 0}\Big\vert\sum_{\substack{\mu\in F \\ {\lambda=\mu{\mu}'\in F}}}a_{\mu,\mu}\Big\vert\;\;\text{since $a_{\mu,\mu}=0$ for $\mu\in\vee\overline{F}\setminus F$}.\\
\end{split}\]
Hence (\ref{eq3}) establishes \eqref{Claim1}.

It follows that the map $t_{\mu}t^{\ast}_{\nu}\mapsto\delta_{\mu,\nu}t_{\mu}t^{\ast}_{\nu}$ extends
to a well-defined linear map $\Phi_t$ from $C^{\ast}(\{t_{\lambda}:\lambda\in\Lambda\})$ to
$\clsp\{q_{\lambda}:\lambda\in\Lambda\}$. Since $\Phi_t$ is a linear idempotent of norm one, it is
a conditional expectation by \cite[Theorem II.6.10.2]{B}. The final statement is straightforward to
check since the two maps in question agree on spanning elements.
\end{proof}

\begin{lemma}\label{pos}
Let $(\Lambda,d)$ be an aperiodic finitely aligned $k$-graph, and let $C^*_{\min}(\Lambda)$ be as in
Definition~\ref{dfn:csmin}. Then the expectation
${\Phi}_{S^{{\ap}}}:C^{\ast}_{{\min}}(\Lambda)\rightarrow\clsp\{P^{{\ap}}_{\lambda}:\lambda\in\Lambda\}$
obtained from Proposition \ref{CE} is faithful on positive elements.
\end{lemma}
\begin{proof}
Let $\{\xi_x:x\in\partial\Lambda^{\ap}\}$ be the canonical orthonormal basis for $\Hh^{\ap} =
{\ell}^2({\partial\Lambda}^{{\ap}})$, and for each $x\in\partial\Lambda^{\ap}$, let
$\theta_{\xi_x,\xi_x}$ be the rank-one projection onto $\mathbb{C}\xi_x$. Since the canonical
expectation $\Phi(T)(h)=\sum_{x}\theta_{\xi_x,\xi_x}T\theta_{\xi_x,\xi_x}(h)$ on
$B(\mathcal{\Hh^{\ap}})$ is faithful, it suffices to show that for $a\in
C^{\ast}_{{\min}}(\Lambda)$, we have
$\Phi_{S^{{\ap}}}(a)=\sum_{x\in{\partial\Lambda}^{{\ap}}}\langle a\xi_x\vert\xi_x\rangle
\theta_{\xi_x,\xi_x}$.

Recall the definition of $\{S^{{\ap}}_{\lambda}:\lambda\in\Lambda\}$ from Proposition \ref{CCK}.
It is not hard to calculate
\[(S^{{\ap}}_{\lambda})^{\ast}\xi_x=\begin{cases}\xi_y\;\;\;\text{if}\;\;x={\lambda}y\;\;\text{for some $y$},\\0\;\;\;\;\text{otherwise.}\end{cases}\]
Fix $\mu,\nu\in\Lambda$ and $x\in{\partial\Lambda}^{{\ap}}$. Then,
\[\langle S^{{\ap}}_{\mu}(S^{{\ap}}_{\nu})^{\ast}\xi_x\vert\xi_x\rangle
=\langle(S^{{\ap}}_{\nu})^{\ast}\xi_x\vert(S^{{\ap}}_{\mu})^{\ast}\xi_x\rangle =\begin{cases}
1\;\;\;\text{if}\;\;x={\mu}y={\nu}y\;\;\text{for some $y$},\\ 0\;\;\;\;\text{otherwise.}\end{cases}\]

Suppose that $x={\mu}y={\nu}y$. Let $m=d(\mu)$ and $n=d(\nu)$. Then Lemma \ref{APP} implies that
${\sigma}^m(x)={\sigma}^m({\mu}y)=y\in{\partial\Lambda}^{{\ap}}$ and
${\sigma}^n(x)={\sigma}^n({\nu}y)=y$. Since $x\in\partial\Lambda^{\ap}$, this forces $m=n$, i.e.
$d(\mu)=d(\nu)$. Since $x={\mu}y={\nu}y$, the factorization property implies $\mu=\nu$. Hence,
\[\sum_{x\in{\partial\Lambda}^{{\ap}}}\langle S^{{\ap}}_{\mu}(S^{{\ap}}_{\nu})^{\ast}\xi_x\vert\xi_x\rangle\theta_{\xi_x,\xi_x}
=\delta_{\mu,\nu}\text{proj}_{\clsp\{\xi_x:x\in{\partial\Lambda}^{{\ap}}\}}
=\delta_{\mu,\nu}P^{{\ap}}_{\mu}=\Phi_{S^{{\ap}}}(S^{{\ap}}_{\mu}(S^{{\ap}}_{\nu})^{\ast}).\] Since
$\clsp\{S^{{\ap}}_{\mu}(S^{{\ap}}_{\nu})^{\ast}:\mu,\nu\in\Lambda\}$ is dense in
$C^{\ast}_{{\min}}(\Lambda)$, we now have
\[
\Phi_{S^{{\ap}}}(a)=\sum_{x\in{\partial\Lambda}^{{\ap}}}\langle a \xi_x\vert\xi_x\rangle\theta_{\xi_x,\xi_x}\text{ for all }a\in C^{\ast}_{{\min}}(\Lambda).\qedhere
\]
\end{proof}

\begin{theorem}\label{main}
Let $(\Lambda, d)$ be an aperiodic finitely aligned $k$-graph. Then
 $C^{\ast}_{min}(\Lambda):=C^{\ast}(S^{{\ap}})$ is co-universal in the
sense that given any Toeplitz-Cuntz-Krieger $\Lambda$-family
$\{t_{\lambda}:\lambda\in\Lambda\}$ in which $t_v$ is nonzero for each $v\in{\Lambda}^0$, there is
a homomorphism $\psi_{t}:C^{\ast}(t)\rightarrow C^{\ast}_{min}(\Lambda)$ such that
$\psi_t(t_{\lambda})=S^{{\ap}}_{\lambda}$ for all $\lambda\in\Lambda$.

The pair $(C^{\ast}_{{\min}}(\Lambda), S^{{\ap}})$ is unique up to canonical isomorphism: if $B$ is
$C^{\ast}$-algebra generated by Toeplitz-Cuntz-Kriger $\Lambda$-family
$\{t_{\lambda}:\lambda\in\Lambda\}$ with the same co-universal property, then there is an
isomorphism $C^{\ast}_{{\min}}(\Lambda)\cong B$ which carries each $S^{{\ap}}_{\lambda}$ to
$t_{\lambda}$.
\end{theorem}
\begin{proof}
Let $\{t_{\lambda}:\lambda\in\Lambda\}$ be a Toeplitz-Cuntz-Krieger $\Lambda$-family
 with each $t_v$ nonzero. We will show that
$\text{ker}(\pi_t)\subset\text{ker}(\pi_{S^{{\ap}}})$ in $\mathcal{T}C^{\ast}(\Lambda)$.

Since each of $t^{\ast}_{\lambda}t_{\lambda}\ne 0$, Propositions \ref{hm} and \ref{CE} imply that
there is a homomorphism $\psi_q$ from $\clsp\{t_{\lambda}t^*_{\lambda}:\lambda\in\Lambda\}$ to
$\clsp\{P^{{\ap}}_{\lambda}:\lambda\in\Lambda\}$ taking each $t_{\lambda}t^*_{\lambda}$ to
$P^{{\ap}}_{\lambda}$. So, we have
\begin{equation*}\begin{split}
\pi_t(a)=0&\Longleftrightarrow\pi_t(a^{\ast}a)=0\\
&\Longrightarrow \psi_q\circ\Phi_t\circ\pi_t(a^{\ast}a)=0\\
&\Longleftrightarrow(\Phi_{S^{{\ap}}}\circ\pi_{S^{{\ap}}})(a^{\ast}a)=0\;\;\text{by (\ref{exp})}\\
&\Longleftrightarrow \pi_{S^{{\ap}}}(a^{\ast}a)=0\;\;\text{by Lemma \ref{pos}}\\
&\Longleftrightarrow \pi_{S^{{\ap}}}(a)=0.
\end{split}
\end{equation*}
Thus, $\text{ker}(\pi_t)\subset\text{ker}(\pi_{S^{{\ap}}})$. Therefore, there exists a well-defined
homomorphism $\psi_t:C^{\ast}(t)\rightarrow C^{\ast}(S^{{\ap}}):=C^{\ast}_{{\min}}(\Lambda)$.

For the final statement of the theorem, we use the same argument as in the proof of \cite[Theorem 3.1]{SW}.
\end{proof}

\begin{remark}
We have denoted both the homomorphism of commutative algebras arising in Proposition~\ref{hm} and
the homomorphism arising in Theorem~\ref{main} by $\psi$. This notation is compatible:
Lemma~\ref{lem2} shows that each Toeplitz-Cuntz-Krieger $\Lambda$-family $\{t_\lambda : \lambda \in
\Lambda\}$ determines a boolean representation $\{q_\lambda : \lambda \in \Lambda\}$ of $\Lambda$,
and then the homomorphism $\psi_t$ agrees upon restriction to $\clsp\{s_\lambda s^*_\lambda :
\lambda \in \Lambda\}$ with $\psi_q$.
\end{remark}

\begin{theorem}\label{thm:injectivity}
Let $(\Lambda, d)$ be an aperiodic finitely aligned $k$-graph with no sources.
\begin{enumerate}
\item If $\{t_{\lambda}:\lambda\in\Lambda\}$ is a Toeplitz-Cuntz-Krieger $\Lambda$-family with
    each $t_v$ nonzero, then the homomorphism $\psi_t$ of Theorem \ref{main} is injective if
    and only if
    \begin{enumerate}
    \item\label{it:CKrel} $\prod_{\lambda\in F}(t_v-t_{\lambda}t^{\ast}_{\lambda})=0$
        whenever $v\in{\Lambda}^0$ and $F\subset v{\Lambda}^0$ is finite exhaustive; and
    \item\label{it:FCE} the expectation $\Phi_t$ is faithful.
    \end{enumerate}
\item If $\phi$ is a homomorphism from $C^{\ast}_{{\min}}(\Lambda)$ to a $C^{\ast}$-algebra $C$
    such that each $\phi(P^{{\ap}}_v)$ is nonzero, then $\phi$ is injective.
\end{enumerate}
\end{theorem}
\begin{proof}
The proof is essentially the same as that of \cite[Theorem 3.2]{SW}.
\end{proof}

\begin{corollary}[The Cuntz-Krieger uniqueness theorem]
Let $(\Lambda,d)$ be an aperiodic finitely aligned $k$-graph, and let $C^{\ast}(\Lambda)$ be the
Cuntz-Krieger algebra of \cite{RSY2}. There is an isomorphism $C^{\ast}(\Lambda)\cong
C^{\ast}_{\text{min}}(\Lambda)$ which carries each $s_{\lambda}$ to $S^{\ap}_{\lambda}$. Moreover,
every Toeplitz-Cuntz-Krieger $\Lambda$-family $\{t_{\lambda}:\lambda\in\Lambda\}$ in which each
$t_v$ is nonzero, and which satisfies condition (\ref{it:CKrel}) of Theorem \ref{thm:injectivity}
generates an isomorphic copy of $C^*(\Lambda)$. In particular, for Toeplitz-Cuntz-Krieger families,
Condition~(\ref{it:CKrel}) of Theorem \ref{thm:injectivity} implies Condition~(\ref{it:FCE}).
\end{corollary}
\begin{proof}
By \cite[Proposition~2.12]{RSY2}, $C^*(\Lambda)$ is generated by a Toeplitz-Cuntz-Krieger family
$\{s_{\lambda}:\lambda\in\Lambda\}$ in which each $s_v$ is nonzero. Hence Theorem~\ref{main}
implies that there is a homomorphism $\psi_s$ from $C^*(\Lambda)$ to $C^*_{{\min}}(\Lambda)$
carrying each $s_\lambda$ to $S^{{\ap}}_\lambda$. The map $\Phi_s$ agrees with that obtained by
first averaging over the gauge-action on $C^*(\Lambda)$ onto its AF core, and then applying the
canonical expectation from the AF algebra onto its diagonal subalgebra. Since each of these
expectations is faithful, so is $\Phi_s$. We have $\prod_{\lambda\in
F}(s_v-s_{\lambda}s^{\ast}_{\lambda})=0$ whenever $v\in{\Lambda}^0$ and $F\subset v{\Lambda}^0$ is
finite exhaustive by (CK).
 Hence Theorem~\ref{thm:injectivity} implies that
$\pi_s$ is an isomorphism.

Now fix a Cuntz-Krieger $\Lambda$-family $\{t_\lambda : \lambda \in \Lambda\}$. The universal
property of $C^*(\Lambda)$ gives a homomorphism $\rho_t : C^*(\Lambda) \to C^*(t)$ satisfying
$\rho_t(s_\lambda) = t_\lambda$. The co-universal property of $C^*_{{\min}}(\Lambda)$ gives a
homomorphism $\psi_t : C^*(t) \to C^*_{{\min}}(\Lambda)$ satisfying $\psi_t(t_\lambda) =
S^{{\ap}}_\lambda$. Since $\psi_s^{-1} \circ \psi_t$ is an inverse for $\rho_t$, the result
follows.
\end{proof}


\begin{thebibliography}{00}


\bibitem{BatesPaskEtAl:NYJM00} T. Bates, D. Pask, I. Raeburn, and W. Szyma{\'n}ski, \emph{The
    {$C\sp *$}-algebras of row-finite graphs}, New York J. Math. \textbf{6} (2000), 307--324.
\bibitem{B}{B. Blackadar}, Operator algebras, Springer, Berlin, 2006.
\bibitem{Cuntz:CMP77} J. Cuntz, \emph{Simple {$C\sp*$}-algebras generated by isometries}, Comm.
    Math. Phys. \textbf{57} (1977), 173--185.
\bibitem{FMY}{C. Farthing, P. Muhly, and T. Yeend}, \emph{Higher-rank graph $C^{\ast}$-algebras : an
    inverse semigroup and groupoid approach}, Semigroup Forum, {\bf 71} (2005), 159--187.
\bibitem{FS1}{N.J. Fowler and A. Sims}, \emph{Product systems over right-angled artin semigroups}, Trans.
    Amer. Math. Soc. {\bf 354} (2002), 1487--1509.
\bibitem{HongSzyma'nski:JMSJ04} J.H. Hong  and W. Szyma{\'n}ski, \emph{The primitive ideal space
    of the {$C\sp \ast$}-algebras of infinite graphs}, J. Math. Soc. Japan \textbf{56} (2004),
    45--64.
\bibitem{Katsura:TAMS04} T. Katsura, \emph{A class of {$C\sp \ast$}-algebras generalizing both
    graph algebras and homeomorphism {$C\sp \ast$}-algebras {I}. {F}undamental results}, Trans.
    Amer. Math. Soc. \textbf{356} (2004), 4287--4322.
\bibitem{K1}{T. Katsura}, \emph{A class of $C^{\ast}$-algebras generalizing both graph algebras and
    homeomorphism $C^{\ast}$-algebras. III. Ideal structures}, Ergodic Theory Dynam. Systems {\bf
    26} (2006), 1805--1854.
\bibitem{K2}{T. Katsura}, \emph{Ideal structure of $C^{\ast}$-algebras associated with
    $C^{\ast}$-correspondences}, Pacific J. Math. {\bf 230} (2007), 107--146.
\bibitem{KP1}{A. Kumjian and D. Pask}, \emph{Higher rank graph $C^{\ast}$-algebras}, New York J. Math.
    {\bf 6} (2000), 1--20.
\bibitem{KPR} A. Kumjian, D. Pask, and I. Raeburn, \emph{Cuntz-{K}rieger algebras
    of directed graphs}, Pacific J. Math. \textbf{184} (1998), 161--174.
\bibitem{KPRR} A. Kumjian, D. Pask, I. Raeburn, and J. Renault, \emph{Graphs,
    groupoids, and {C}untz-{K}rieger algebras}, J. Funct. Anal. \textbf{144} (1997), 505--541.
\bibitem{LS}{P. Lewin and A. Sims}, \emph{Aperiodicity and cofinality for finitely aligned higher-rank
    graphs}, Math. Proc. Cambridge Philos. Soc. {\bf 149} (2010), 333--350.
\bibitem{Raeburn:cbms} I. Raeburn, Graph algebras, Published for the Conference Board of
    the Mathematical Sciences, Washington, DC, 2005, vi+113.
\bibitem{RSY2}{I. Raeburn, A. Sims, and T. Yeend}, \emph{The $C^{\ast}$-algebras of finitely aligned
    higher-rank graphs}, J. Funct. Anal. {\bf 213} (2004), 206--240.
\bibitem{RS1}{I. Raeburn and A. Sims}, \emph{Product systems of graphs and the Toeplitz algebras of
    higher-rank graphs}, J. Operator Theory. {\bf 53}, (2005), 399--429.
\bibitem{RoS1}{D.I. Robertson and A. Sims}, \emph{Simplicity of $C^{\ast}$-algebras associated to
    higher-rank graphs}, Bull. London Math. Soc. {\bf 39} (2007), 337--344.
\bibitem{Sims:IUMJ06} A. Sims, \emph{Relative {C}untz-{K}rieger algebras of
    finitely aligned higher-rank graphs}, Indiana Univ. Math. J. \textbf{55} (2006), 849--868.

\bibitem{SW}{A. Sims and S.B.G. Webster}, \emph{A direct approach to co-universal algebras
    associated
    to directed graphs}, Bull. Malays. Math. Sci. Soc. (2) {\bf 33} (2010), 211--220.
\end{thebibliography}
\end{document}